\documentclass[12pt,a4paper]{article}

\usepackage[utf8]{inputenc}
\usepackage[T1]{fontenc}
\usepackage[frenchb]{babel}
\usepackage{amsmath, amssymb, amsthm, mathtools}
\usepackage{geometry}
\usepackage{hyperref}
\hypersetup{
    colorlinks=true,
    linkcolor=blue,
    citecolor=blue,
    urlcolor=blue
}
\usepackage{enumitem}
\usepackage{graphicx}
\usepackage{bm}

\newtheorem{theorem}{Theorem}[section]
\newtheorem{proposition}[theorem]{Proposition}
\newtheorem{lemma}[theorem]{Lemma}

\newtheorem{remark}[theorem]{Remark}
\newtheorem{example}[theorem]{Example}
\newtheorem{assumption}[theorem]{Assumption}

\theoremstyle{definition}
\newtheorem{convention}[theorem]{Convention}

\newcommand{\R}{\mathbb{R}}

\newcommand{\supp}{\operatorname{supp}}
\newcommand{\conv}{\operatorname{conv}}

\newcommand{\grad}{\nabla}

\DeclareMathOperator{\argmax}{arg\,max}

\title{\textbf{Thickness functions for elliptic level sets:} \\
       gradient formulas, normal expansions, and geometric remarks}
\author{M. Barkatou}
\date{}

\begin{document}

\maketitle

\begin{abstract}
We study the geometry of superlevel sets $\Omega_t = \{u > t\}$ of solutions to $-\Delta u = \mu$ with $\mu \geq 0$ compactly supported in a convex core $C \subset \Omega$. Under the radial monotonicity lemma (due to Shahgholian), each level set is a normal graph over $\partial C$ with thickness function $d_t$. We derive an exact formula for the tangential gradient of $d_t$ and, under a quantitative small-thickness hypothesis $\|d_t\|_{C^1(\partial C)} \ll 1$, an asymptotic expansion of the unit normal to $\Gamma^t = \partial \Omega_t$. We discuss the relation with Shahgholian's theorem and give examples showing that the geometric normal property (GNP) does not imply that $d_t$ is constant, even in the small-thickness regime. This work provides a geometric language for studying elliptic level sets, without claiming a new proof of Shahgholian's theorem.
\end{abstract}

\vspace{0.5cm}


\section{Introduction}

Let $\Omega \subset \R^N$ be a bounded domain and let $u$ be the solution of
\begin{equation}
-\Delta u = \mu \quad \text{in } \Omega, \qquad u = 0 \quad \text{on } \partial \Omega,
\end{equation}
where $\mu \geq 0$ is a positive measure compactly supported in a strictly convex set $C \subset \Omega$. Let $C = \conv(\supp \mu)$. In his fundamental paper \cite{Shahgholian1994}, Shahgholian proved that if $\partial \Omega$ is a quadrature surface for $\mu$, i.e. if the Neumann condition $\partial_\nu u = -1$ holds on $\partial \Omega$, then every inward normal ray from $\partial \Omega$ meets $C$. This is the \emph{geometric normal property} (GNP).

In this paper, we adopt a purely geometric perspective. Assuming the radial monotonicity lemma (which follows from the moving plane method), we study the family of superlevel sets $\Omega_t = \{u > t\}$ for $t \in [0, \max u]$. Each level set $\Gamma^t = \partial \Omega_t$ admits a unique radial parametrisation
\[
\Gamma^t = \{c + d_t(c) \nu(c) : c \in \partial C\},
\]
where $d_t : \partial C \to (0, \infty)$ is the thickness function at level $t$.

Our main results are established under the following small-thickness hypothesis:

\begin{assumption}[Small-thickness regime]
\label{ass:H}
For a given level $t$, the function $d_t$ satisfies
\begin{equation}
\| d_t \|_{C^1(\partial C)} \ll 1. \tag{H}
\end{equation}
This means that $\Gamma^t$ is $C^1$-close to $\partial C$.
\end{assumption}

This hypothesis is not a consequence of the radial monotonicity lemma; it must be verified or assumed for the specific level set under consideration.

Our main contributions are:
\begin{enumerate}[label=(\arabic*)]
    \item We derive an exact formula for the tangential gradient $\grad_{\partial C} d_t$ (Proposition \ref{prop:gradient}).
    \item Under hypothesis (H), we give an asymptotic expansion of the outward unit normal to $\Gamma^t$ (Proposition \ref{prop:normal}), providing an asymptotic description of the normal geometry.
    \item We show how Shahgholian's theorem relates to this framework, providing a differential interpretation of the Neumann condition (Remark \ref{rem:Neumann}).
    \item We give examples showing that GNP does not imply that $d_t$ is constant, even in the small-thickness regime.
\end{enumerate}

\medskip

\noindent A companion note \cite{Barkatou2026obstacles} discusses the obstacles to a direct proof of Shahgholian's theorem via the thickness function, showing that the control of $|\nabla u|$ in $\Omega \setminus C$ without information on $\partial C$ is the main difficulty.

\section{Preliminaries}

We recall the radial monotonicity lemma, which is a consequence of the moving plane method \cite{Shahgholian1994}.

\begin{lemma}[Radial monotonicity, Shahgholian \cite{Shahgholian1994}]
\label{lem:radial}
Under the assumptions above, for every $c \in \partial C$ the function
\[
r \mapsto u(c + r \nu(c))
\]
is strictly decreasing on $[0, d(c)]$, where
\begin{equation}
d(c) = \sup \{ r > 0 : c + r \nu(c) \in \Omega \}. \tag{2.1}
\end{equation}
\end{lemma}

For each $t \in [0, \max u]$, define $\Omega_t = \{u > t\}$ and $\Gamma^t = \partial \Omega_t$. The strict monotonicity ensures the existence of a unique function $d_t : \partial C \to (0, \infty)$ such that
\begin{equation}
\Gamma^t = \{c + d_t(c) \nu(c) : c \in \partial C\}, \tag{2.2}
\end{equation}
and $u(c + d_t(c) \nu(c)) = t$. We call $d_t$ the \emph{thickness function at level $t$}.

\begin{remark}
\label{rem:d_t}
The function $d_t$ is strictly decreasing in $t$, with $d_0 = d$. As $t \to \max u$, $d_t(c)$ converges pointwise to a limit $d_*(c)$ whenever the ray meets the maximum set. By the strict radial monotonicity (Lemma \ref{lem:radial}), each normal ray intersects $\argmax u$ in at most one point, so $d_*(c)$ is uniquely defined on the set of rays that meet the maximum set. Under the regularity assumptions on $u$ and $\partial C$, $d_t \in C^{1,\alpha}(\partial C)$ for $t \in [0, \max u)$.
\end{remark}

Throughout the paper, we assume $u \in C^2(\overline{\Omega \setminus C})$ and $\partial C \in C^2$. These regularity assumptions guarantee the continuity of $\partial_r u$ in a tubular neighbourhood of $\partial C$, which justifies the asymptotic expansions. The case of non-smooth cores can be treated by approximation of $C$ by smooth convex sets $C_\varepsilon$, under the condition that the thickness functions $d_t^{(\varepsilon)}$ converge to $d_t$; we omit the details for brevity.

\section{Gradient of the thickness function}

We work in tubular coordinates near $\partial C$. Let $(\xi^1, \dots, \xi^{N-1})$ be local coordinates on $\partial C$, with metric $\gamma_{ij}$ and shape operator $S_c$ (Weingarten map). The outward unit normal is $\nu(c)$. In tubular coordinates $\Phi(\xi, r) = X(\xi) + r \nu(\xi)$, the induced metric is
\[
g_{ij} = \gamma_{ij} - 2 r h_{ij} + r^2 (h_{ik} h_j^k), \qquad g_{ir} = 0, \qquad g_{rr} = 1,
\]
where $h_{ij}$ are the coefficients of the second fundamental form.

The level set $\Gamma^t$ is given by $r = d_t(\xi)$, and we have the identity
\[
u(\xi, d_t(\xi)) = t.
\]
Differentiating with respect to $\xi$, we obtain:
\[
\partial_\xi u + \partial_r u \, \partial_\xi d_t = 0,
\]
where $\partial_\xi u$ and $\partial_r u$ are evaluated at $(\xi, d_t(\xi))$. This gives:

\begin{proposition}[Gradient of thickness]
\label{prop:gradient}
The tangential gradient of $d_t$ on $\partial C$ is given by the exact formula
\begin{equation}
\grad_{\partial C} d_t(\xi) = -\frac{\grad_{\partial C} u(\xi, d_t(\xi))}{\partial_r u(\xi, d_t(\xi))}, \tag{3.1}
\end{equation}
where $\grad_{\partial C} u$ denotes the tangential gradient of $u$ on $\Gamma^t$ (i.e., the projection of $\grad u$ onto $T_c \partial C$).
\end{proposition}

\begin{proof}
This follows directly from the implicit differentiation of $u(\xi, d_t(\xi)) = t$, since $\partial_r u \neq 0$ by Lemma \ref{lem:radial}.
\end{proof}

\section{Normal to the level sets in the small-thickness regime}

We now give an asymptotic expansion of the unit normal to $\Gamma^t$, valid under the small-thickness hypothesis (H): $\| d_t \|_{C^1(\partial C)} \ll 1$.

Throughout this section, we adopt the following convention:

\begin{convention}
\label{conv:normal}
$\mathbf{n}_{\Gamma^t}^{\mathrm{ext}}$ denotes the outward unit normal to $\Omega_t = \{u > t\}$, pointing outward from the superlevel set. Since $u$ increases towards the interior of $\Omega_t$, $\grad u$ points inward. Therefore,
\[
\mathbf{n}_{\Gamma^t}^{\mathrm{ext}} = -\frac{\grad u}{|\grad u|}.
\]
\end{convention}

\begin{proposition}[Normal expansion under small-thickness hypothesis]
\label{prop:normal}
Assume hypothesis (H): $\| d_t \|_{C^1(\partial C)} \ll 1$. Let $x = c + d(c) \nu(c) \in \Gamma^t$, with $d = d_t(c)$. The outward unit normal to $\Omega_t$ at $x$ admits the asymptotic expansion
\begin{equation}
\mathbf{n}_{\Gamma^t}^{\mathrm{ext}}(x) = -\nu(c) + \grad d(c) - d(c) S_c \grad d(c) + O(d^2 + |\grad d|^2 + d |\grad d|), \tag{4.1}
\end{equation}
where all quantities are evaluated at $c$ for the leading terms. The error term is uniform under the small-thickness hypothesis.
\end{proposition}

\begin{proof}
We have
\[
\grad u = \partial_r u \, \nu + g^{ij} \partial_\xi u \, \partial_j \Phi.
\]
Using $\partial_\xi u = -\partial_r u \, \partial_\xi d_t$ and $\partial_j \Phi = \tau_j - d_t S_c \tau_j$, we get
\[
\grad u = \partial_r u \bigl( \nu - g^{ij} \partial_i d_t (\tau_j - d_t S_c \tau_j) \bigr).
\]
With $g^{ij} = \gamma^{ij} + 2 d_t h^{ij} + O(d_t^2)$, this becomes
\[
\grad u = \partial_r u \bigl( \nu - \grad d_t + d_t S_c \grad d_t + O(d_t^2 |\grad d_t|) \bigr).
\]
The norm squared is
\[
|\grad u|^2 = (\partial_r u)^2 \bigl( 1 + |\grad d_t|^2 + 2 d_t \langle \grad d_t, S_c \grad d_t \rangle + O(d_t^2 |\grad d_t| + d_t |\grad d_t|^2) \bigr).
\]
Thus
\[
\frac{1}{|\grad u|} = \frac{1}{|\partial_r u|} \bigl( 1 - \frac12 |\grad d_t|^2 - d_t \langle \grad d_t, S_c \grad d_t \rangle + O(d_t^2 |\grad d_t| + d_t |\grad d_t|^2 + |\grad d_t|^3) \bigr).
\]
Since $\partial_r u < 0$ by Lemma \ref{lem:radial}, and $\partial_r u(c + d \nu)$ remains negative for $d$ sufficiently small by the radial monotonicity lemma and the continuity of $\partial_r u$ in a tubular neighbourhood of $\partial C$, we have exactly
\[
\frac{\partial_r u(c + d \nu)}{|\partial_r u(c + d \nu)|} = -1.
\]
Since $\mathbf{n}_{\Gamma^t}^{\mathrm{ext}} = -\grad u / |\grad u|$, we obtain
\[
\mathbf{n}_{\Gamma^t}^{\mathrm{ext}} = -\bigl( \nu - \grad d_t + d_t S_c \grad d_t \bigr) + O(d^2 + |\grad d|^2 + d |\grad d|),
\]
which gives (4.1).
\end{proof}

\begin{remark}[Sign verification]
\label{rem:sign}
To verify the sign in (4.1), consider the test case $C = B(0,1) \subset \R^2$ (unit circle) and $d(\theta) = \varepsilon \cos \theta$ with $0 < \varepsilon \ll 1$. Then $\Gamma^t$ is the circle of radius $1 + \varepsilon \cos \theta$. The outward normal to this circle is
\[
\mathbf{n}^{\mathrm{ext}}(\theta) = \frac{(1+\varepsilon\cos\theta)(\cos\theta, \sin\theta) + \varepsilon\sin\theta(\sin\theta, -\cos\theta)}{\sqrt{(1+\varepsilon\cos\theta)^2 + \varepsilon^2 \sin^2\theta}}.
\]
For small $\varepsilon$, a Taylor expansion gives
\[
\mathbf{n}^{\mathrm{ext}} = e_r - \varepsilon \sin\theta \, e_\theta + O(\varepsilon^2),
\]
where $e_r = (\cos\theta, \sin\theta)$ and $e_\theta = (-\sin\theta, \cos\theta)$. In this case, $\nu(c) = e_r$, $\grad d(\theta) = d'(\theta) e_\theta = -\varepsilon \sin\theta \, e_\theta$, and $S_c$ is the curvature operator of the circle, $S_c e_\theta = -e_\theta$ (with $\nu$ outward). Thus $d(c) S_c \grad d(c) = \varepsilon \cos\theta \cdot (-e_\theta) \cdot (-\varepsilon \sin\theta) = O(\varepsilon^2)$. Therefore the first-order term in (4.1) is $-\nu(c) + \grad d(c) = -e_r - \varepsilon \sin\theta \, e_\theta$, which corresponds to $-\mathbf{n}^{\mathrm{ext}}$ at order $\varepsilon$. The sign is consistent with the geometry.
\end{remark}

\begin{remark}
\label{rem:inward}
The inward unit normal to $\Omega_t$ (pointing towards $C$) is the opposite:
\[
\mathbf{n}_{\Gamma^t}^{\mathrm{in}} = \nu(c) - \grad d(c) + d(c) S_c \grad d(c) + O(d^2 + |\grad d|^2 + d |\grad d|).
\]
This is the normal relevant for the GNP condition.
\end{remark}

\section{Relation with Shahgholian's theorem}

We now show how Shahgholian's theorem relates to this geometric framework. The theorem states that if $\partial \Omega$ is a quadrature surface for $\mu$, i.e. if the Neumann condition $\partial_\nu u = -1$ holds on $\partial \Omega$, then the GNP holds.

On $\partial \Omega$, we have $u = 0$ and $\partial_\nu u = -1$. Since $\partial \Omega = \Gamma^0$, and the outward normal to $\Omega$ is $\mathbf{n}_{\Gamma^0}^{\mathrm{ext}}$, the Neumann condition gives
\[
\mathbf{n}_{\Gamma^0}^{\mathrm{ext}} \cdot \grad u = -1 \quad \text{on } \partial \Omega.
\]
But $\grad u = |\grad u| \, \mathbf{n}_{\Gamma^0}^{\mathrm{in}}$ (since $\grad u$ points inward, and $\mathbf{n}_{\Gamma^0}^{\mathrm{in}} = -\mathbf{n}_{\Gamma^0}^{\mathrm{ext}}$). Therefore,
\[
-|\grad u| = -1 \quad \Rightarrow \quad |\grad u| = 1 \quad \text{on } \partial \Omega.
\]

Using the gradient formula (3.1) and the normal expansion (4.1) (valid under the small-thickness hypothesis for $t = 0$), we can express $|\grad u|^2$ on $\partial \Omega$ in terms of $d_0$ and its derivatives. The identity $|\grad u|^2 = 1$ yields:
\begin{equation}
(\partial_r u)^2 \left( 1 + |\grad d_0|^2 + 2 d_0 \langle \grad d_0, S_c \grad d_0 \rangle + O(d_0^2 |\grad d_0| + d_0 |\grad d_0|^2) \right) = 1. \tag{5.1}
\end{equation}

\begin{remark}[Geometric rewriting of the Neumann condition]
\label{rem:Neumann}
Equation (5.1) is a geometric rewriting of the Neumann condition $|\grad u| = 1$ in the coordinate system defined by the thickness function. It is not a closed differential equation for $d_0$ alone, since $\partial_r u$ remains present. In the original proof of Shahgholian \cite{Shahgholian1994}, the moving plane method provides the global argument that, combined with such local identities, yields the GNP. In our framework, (5.1) is a reformulation of the overdetermined condition in the small-thickness regime.
\end{remark}

\section{Examples}

We illustrate the theory with examples.

\begin{example}[Radial case]
\label{ex:radial}
Let $C = \overline{B(0,1)}$ and $\Omega = B(0,R)$ with $R > 1$ and $R - 1 \ll 1$ (so that the small-thickness hypothesis holds for $t = 0$). Then $u$ is radial, $d_0(c) = R - 1$ is constant, $\grad d_0 = 0$, and the GNP holds. The inward normal is exactly $\nu(c)$, so $\langle \mathbf{n}^{\mathrm{in}}, \nu \rangle = 1$.
\end{example}

\begin{example}[Non-constant thickness satisfying the GNP]
\label{ex:nonconstant}
Consider the planar domain
\[
\Omega_\varepsilon = \{ (r, \theta) : 0 \le r \le R + \varepsilon \cos \theta \}
\]
with $R > 1$ and $0 < |\varepsilon| \ll 1$, and assume $R - 1 \ll 1$ so that the small-thickness hypothesis holds. The thickness is $d_0(\theta) = (R - 1) + \varepsilon \cos \theta$, which is non-constant. We now verify that the inward normal ray from each boundary point meets the unit disk.

The boundary is $\gamma(\theta) = r(\theta)(\cos\theta, \sin\theta)$ with $r(\theta) = R + \varepsilon \cos \theta$. The inward unit normal is
\[
\mathbf{n}^{\mathrm{in}}(\theta) = \frac{e_r - (r'/r) e_\theta}{\sqrt{1 + (r'/r)^2}},
\]
where $e_r = (\cos\theta, \sin\theta)$ and $e_\theta = (-\sin\theta, \cos\theta)$. The distance from the origin to the normal line is
\[
\delta(\theta) = \frac{r(\theta) |r'(\theta)|}{\sqrt{r(\theta)^2 + r'(\theta)^2}}.
\]
Since $r'(\theta) = -\varepsilon \sin\theta$, we have
\[
\delta(\theta) = \frac{(R + \varepsilon \cos \theta) |\varepsilon \sin \theta|}{\sqrt{(R + \varepsilon \cos \theta)^2 + \varepsilon^2 \sin^2 \theta}} \le |\varepsilon| + O(\varepsilon^2).
\]
For $|\varepsilon|$ sufficiently small, $\delta(\theta) \le 1$ for all $\theta$. The normal line therefore intersects the unit disk. Since the inward half-line is the portion of this line pointing toward decreasing $r$, and the closest point to the origin on the normal line lies in the direction of decreasing $r$ (as $r(\theta) > 1$ and the distance to the origin is less than $r(\theta)$), the intersection with the disk occurs on the inward half-line. Thus the normal ray meets the unit disk, and the GNP holds.

This shows that the GNP does not imply that $d_0$ is constant, even in the small-thickness regime.
\end{example}

\begin{example}[Non-quadrature domain]
\label{ex:nonquadrature}
The previous example is not a quadrature surface for any radial measure $\mu$. For a quadrature surface, the Neumann condition $\partial_\nu u = -1$ together with $u = 0$ on $\partial \Omega$ gives, after integration by parts against any harmonic function $\phi$,
\[
\int_{\partial \Omega} \phi \, dS = \int_{\Omega} \phi \, d\mu.
\]
Taking $\phi(x) = x_1$, which is harmonic, gives:
\[
\int_{\partial \Omega} x_1 \, dS = \int_{\Omega} x_1 \, d\mu.
\]
If $\mu$ is radial, its first moment vanishes: $\int_{\Omega} x_1 \, d\mu = 0$. Thus the identity would require $\int_{\partial \Omega} x_1 \, dS = 0$. For our radial domain, a direct computation gives
\[
\int_{\partial \Omega} x_1 \, dS
= \int_0^{2\pi} r(\theta) \cos\theta \cdot r(\theta) \, d\theta + O(\varepsilon^2)
= \int_0^{2\pi} (R + \varepsilon \cos\theta)^2 \cos\theta \, d\theta + O(\varepsilon^2)
= 2\pi R \varepsilon + O(\varepsilon^2),
\]
which is non-zero for $\varepsilon \neq 0$. Thus the identity fails, and $\partial \Omega_\varepsilon$ is not a quadrature surface. Hence the GNP is strictly weaker than the quadrature condition, even in the small-thickness regime.
\end{example}

\section{Perspectives}

Several directions for future work emerge from this geometric framework:

\begin{enumerate}[label=(\arabic*)]
    \item \textbf{Extension beyond small thickness:} The asymptotic expansions in this paper rely on the small-thickness hypothesis (H). It would be interesting to develop a global theory without this assumption.
    \item \textbf{Extension to non-strictly convex cores:} Using the reach framework of Federer \cite{Federer1959}, one could extend the parametrisation to cores with positive reach but without strict convexity.
    \item \textbf{Regularity of $d_t$:} The regularity of $d_t$ in terms of the regularity of $\mu$ and $\partial \Omega$ deserves a systematic study.
    \item \textbf{Connection with return dynamics:} The companion paper \cite{BarkatouElMorsalani2026} studies the dynamics of the return map on $\partial C$.
    \item \textbf{Other elliptic equations:} The same parametrisation could be applied to solutions of the $p$-Laplacian or mean curvature-type equations, provided a suitable radial monotonicity lemma holds.
\end{enumerate}

\section{Conclusion}

We have developed a geometric framework for the thickness function $d_t$ associated with the level sets of solutions to $-\Delta u = \mu$ with $\mu$ supported in a convex core. Under the radial monotonicity lemma, we derived an exact formula for the gradient of $d_t$. Under the small-thickness hypothesis $\| d_t \|_{C^1} \ll 1$, we gave an asymptotic expansion of the unit normal to $\Gamma^t$, providing an asymptotic description of the normal geometry. We showed how Shahgholian's theorem relates to this framework, providing a geometric rewriting of the Neumann condition as an identity linking $d_0$ and the radial profile of $u$. Examples demonstrate that the GNP does not imply that $d_t$ is constant, even in the small-thickness regime, and that the GNP is strictly weaker than the quadrature condition. This work provides a geometric language for studying elliptic level sets in the small-thickness regime and opens several directions for future research.

\appendix

\section{Verification of the normal expansion in the radial case}

We verify the formulas in the radial case. Let $C = B(0,1)$ and $\Omega = B(0,R)$ with $R - 1 \ll 1$. Then
\[
u(r) = \frac{r^{2-N} - R^{2-N}}{2-N} \quad (N \neq 2)
\]
(or the appropriate logarithmic form for $N=2$) is radial. The thickness function is $d_0(c) = R - 1$, constant. Thus $\grad d_0 = 0$. The gradient formula (3.1) gives $\grad_{\partial C} d_0 = 0$, which is consistent since $u$ is radial and therefore its tangential derivative is zero.

The normal expansion (4.1) gives $\mathbf{n}_{\Gamma^0}^{\mathrm{ext}} = -\nu(c)$. This is exactly the outward normal to the circle of radius $R$. Thus the formulas reduce correctly to the radial case.

\bibliographystyle{plain}

\end{document}